\documentclass[11pt,a4paper]{article}
\usepackage[T1]{fontenc}
\usepackage[utf8]{inputenc}
\usepackage{lmodern}
\usepackage[a4paper,margin=1in]{geometry}
\usepackage{amsmath,amssymb,amsthm,mathrsfs}
\usepackage{microtype}
\usepackage{xcolor}
\usepackage{aliascnt}
\usepackage[hidelinks]{hyperref}
\usepackage[noabbrev,nameinlink]{cleveref}
\numberwithin{equation}{section}
\newtheorem{theorem}{Theorem}[section]
\newaliascnt{lemma}{theorem}
\newtheorem{lemma}[lemma]{Lemma}
\aliascntresetthe{lemma}
\newaliascnt{proposition}{theorem}
\newtheorem{proposition}[proposition]{Proposition}
\aliascntresetthe{proposition}
\newaliascnt{corollary}{theorem}
\newtheorem{corollary}[corollary]{Corollary}
\aliascntresetthe{corollary}
\theoremstyle{remark}
\newaliascnt{remark}{theorem}

\aliascntresetthe{remark}
\crefname{theorem}{Theorem}{Theorems}
\crefname{lemma}{Lemma}{Lemmas}
\crefname{proposition}{Proposition}{Propositions}
\crefname{corollary}{Corollary}{Corollaries}
\crefname{remark}{Remark}{Remarks}
\crefname{section}{Section}{Sections}
\newcommand{\step}[1]{\par\smallskip\noindent\textit{#1}\enspace}
\allowdisplaybreaks[2]
\hypersetup{pdftitle={Observability from Measurable Sets for One-Dimensional Heat Equations with Bounded Spacetime Potentials}}
\title{Observability from Measurable Sets for One-Dimensional Heat Equations with Bounded Spacetime Potentials\thanks{This work was partially supported by the National Natural Science Foundation of China under Grant No.~12601883.}}
\author{Xiaomin Zhu\thanks{Civil Aviation Flight University of China, Deyang 618307, China (xiaominzhu123@yeah.net).}}
\date{}
\begin{document}
\maketitle

\begin{abstract}
This paper study observability for the Dirichlet heat equation on a bounded interval
with a real bounded potential depending on space and time. We prove an
observability inequality from every measurable subset of spacetime with
positive measure. The constant is uniform over potentials with a prescribed
$L^\infty$ bound. By duality, positive spacetime measure is equivalent to null
controllability by square-integrable distributed controls. Controls can also
be chosen bounded in time with values in $L^2$. If almost every spatial slice
of the observation set has a fixed positive lower bound on its measure, the
observability constant is bounded above by $C_0e^{C_1/T}$, uniformly in the
locations of the slices. The proof combines observation estimates on
finite-dimensional spaces transported by the evolution with decay estimates
for the distance to those spaces.
\end{abstract}
\par

\section{Introduction}\label{sec:introduction}

A zeroth-order term in a heat equation can describe local absorption
or growth in a diffusive medium. When the local rate varies with the
position and changes during the evolution, the corresponding
coefficient depends on both space and time. Such spacetime-dependent
potentials arise, for example, in heat-conduction models involving
moving inhomogeneities; see \cite{ChowItoZou2018}. Motivated by such models, we study a heat equation with a general
bounded spacetime potential.

Let $I=(0,L)$, where $L>0$, and let $T>0$. We consider
\begin{equation}\label{eq:controlled}
 \begin{cases}
 \partial_t y-\partial_x^2y+b(t,x)y=\mathbf{1}_D u
     &\text{in }(0,T)\times I,\\
 y(t,0)=y(t,L)=0 &\text{for }t\in(0,T),\\
 y(0)=y_0 &\text{in }I.
 \end{cases}
\end{equation}
Here $b\in L^\infty((0,T)\times I;\mathbb R)$,
$D\subset(0,T)\times I$ is measurable, and
$y_0\in L^2(I;\mathbb R)$.
The control $u\in L^2(D)$ is extended by zero outside $D$.
We impose no regularity on $b$ beyond boundedness and measurability
in the two variables.

How much of spacetime must be observed to recover the final state of
such a heat evolution? By control duality, the same question asks
where a distributed control can act while still driving every
initial state to zero. For the free heat equation, observability is
known from arbitrary spacetime measurable sets of positive measure,
while parabolic equations with spacetime-dependent lower-order
coefficients can be observed from a measurable set of times and a
fixed open spatial set. These results lead naturally to the question
studied here: does positive spacetime measure remain sufficient when
the potential is merely bounded and measurable in both space and
time?

This question combines two difficulties. The observation set may
have spatial slices with empty interior and need not contain a
product of two positive-measure sets. At the same time, a general
bounded potential \(b(t,x)\) gives no spatial analyticity and no fixed
spectral decomposition that is preserved by the evolution.

Two lines of previous work are particularly relevant to this question.
The first concerns observability for parabolic equations with
time-dependent lower-order terms.
Phung and Wang \cite{PW} considered
\[
 \partial_t u-\Delta u+a(x,t)u+\beta(x,t)\cdot\nabla u=0
\]
on bounded convex domains. Their zeroth-order coefficient may depend
on both space and time, and they proved observability from
$E\times\omega$, where $E\subset(0,T)$ has positive measure and
$\omega$ is a fixed nonempty open subset of the spatial domain.
Phung et al.~\cite[Theorem~4]{PWZ2014} also established
measurable-time observability for heat equations with bounded
spacetime potentials, with explicit dependence on the potential norm.
More recently, Zhu and Zhuge \cite{ZZ} studied the dependence of
observability constants on time-dependent potentials and obtained
refined estimates for Lipschitz spacetime potentials.
In these time-dependent results, the spatial observation region
remains open.

The second line concerns observation from measurable sets.
For the free heat equation, Zhang \cite{Zhang2012} proved
observability from a product of measurable sets of positive measure.
Apraiz et al.~\cite{AEWZ} went further
and established observability from an arbitrary measurable subset
of spacetime with positive measure. Thus, for the free heat equation,
no product structure of the observation set is required.
For parabolic equations with analytic coefficients, Escauriaza et al.~\cite{EMZ2015,EMZ2017} obtained measurable-set
observability from quantitative analyticity estimates.
For equations with a potential, measurable spatial observation sets
are also available in some autonomous settings.
In particular, Zhu and Zhuge \cite[Theorem~1.5]{ZZ} proved in one
dimension an observability inequality from $E\times\omega$, with
both $E$ and $\omega$ of positive measure, for bounded
time-independent potentials $b=b(x)$.
Related measurable-set results for stationary operators can be found
in \cite{AE2008,AE2013,BM2023,LBM}.

The present paper combines these two directions.
We keep the general spacetime geometry of
\cite{AEWZ}, while allowing a real potential
$
 b=b(t,x)\in L^\infty((0,T)\times I)
$
with arbitrary measurable dependence on both variables.
More precisely, we prove that every measurable set
$D\subset(0,T)\times I$ with $|D|>0$ is an observation set for the
one-dimensional heat equation with such a potential.
For every $M\geq0$, the observability constant can be chosen uniformly
for all $b$ satisfying $\|b\|_{L^\infty}\leq M$.
By duality, this yields null controllability of
\eqref{eq:controlled} from every positive-measure spacetime set.
We also obtain controls in $L^\infty(0,T;L^2(I))$.
If, in addition, the spatial slices
\[
 D_t=\{x\in I:(t,x)\in D\}
\]
satisfy $|D_t|\geq m>0$ for almost every $t$, the observability
constant admits an upper bound of the form
$
 C_0 e^{C_1/T},
$
uniformly in the locations and measurable motion of the slices.

The main difficulty is caused by the interaction between the
spacetime potential and the spectral decomposition.
For a time-independent potential $b=b(x)$, the spectral subspaces
of the operator $-\partial_x^2+b(x)$ are preserved by the
evolution. Observation on low frequencies can then be combined
with the decay of the complementary high frequencies.
This structure is lost for a general $b(t,x)$.
Indeed, multiplication by $b(t,x)$ transfers energy between fixed
Dirichlet modes. Hence the equation for the high-frequency
component contains a forcing term generated by the low-frequency
component, and the usual high-frequency decay estimate is no longer
closed.
One might instead use the spectral subspaces of
$-\partial_x^2+b(t,\cdot)$ at each time, but these subspaces need not
be compatible with the evolution. Moreover, the assumption
$b\in L^\infty$ gives no time regularity with which to differentiate
the corresponding spectral projections.

To overcome this difficulty, we replace fixed spectral subspaces by
finite-dimensional subspaces transported by the actual evolution.
Starting from the span of the first $N$ Dirichlet eigenfunctions,
we evolve this space with the equation itself.
The resulting moving subspaces are compatible with the
nonautonomous dynamics by construction and require no time
derivatives of $b$.
We prove a measurable-set sampling estimate on these spaces, with
cost of order $e^{CN}$, and a decay estimate of order
$e^{-cN^2(t-s)}$ for the distance of a solution to the moving
subspace.
Balancing these two estimates gives a one-time interpolation
inequality.
A spatial slicing argument and the telescoping series method of
\cite{PW,AEWZ,WZ2017} then give observability from the full
spacetime set $D$.

The rest of the paper is organized as follows.
\cref{sec:evolution} states the main results and introduces
the evolution framework.
\cref{sec:subspaces} establishes the observation and decay
estimates for the moving subspaces.
\cref{sec:observability} combines these estimates to prove
observability and null controllability, and ends with examples of
measurable observation sets.

\section{Main results and the moving-subspace framework}\label{sec:evolution}
We state the main observability and controllability results, then
introduce the evolution family and the moving subspaces used in their
proofs.

Throughout, $|\cdot|$ denotes Lebesgue measure. Set
\[
H=L^2(I;\mathbb R),\qquad V=H_0^1(I).
\]
We say that \eqref{eq:controlled} is null-controllable on $D$ if, for every
$y_0\in H$, there exists a control $u\in L^2(D)$ such that $y(T)=0$.
The corresponding adjoint equation is
\begin{equation}\label{eq:adjoint}
 \begin{cases}
 -\partial_t p-\partial_x^2p+b(t,x)p=0
       &\text{in }(0,T)\times I,\\
 p(t,0)=p(t,L)=0 &\text{for }t\in(0,T),\\
 p(T)=q &\text{in }I,
 \end{cases}
 \qquad q\in H.
\end{equation}

\begin{theorem}\label{thm:main}
Let $D\subset(0,T)\times I$ be measurable and let $M\geq0$.
If $|D|>0$, then there exists a constant
$
C=C(L,T,M,D)>0
$
such that, for every real-valued
$b\in L^\infty((0,T)\times I)$ with
$\lVert b\rVert_{L^\infty((0,T)\times I)}\leq M$,
every solution of \cref{eq:adjoint} satisfies
\begin{equation}\label{eq:observability}
 \lVert p(0)\rVert_H^2
 \leq C\int_D |p(t,x)|^2\,\mathrm{d}t\,\mathrm{d}x.
\end{equation}
Consequently, for every $y_0\in H$, \cref{eq:controlled} admits a control
$u\in L^2(D)$ such that
\begin{equation}\label{eq:control-bound}
 y(T)=0,\qquad
 \lVert u\rVert_{L^2(D)}
 \leq C^{1/2}\lVert y_0\rVert_H.
\end{equation}

If $|D|=0$, then \cref{eq:controlled} is not null-controllable for any
real-valued bounded potential $b$. Hence, for every fixed
$b\in L^\infty((0,T)\times I;\mathbb R)$,
null controllability on $D$ is equivalent to $|D|>0$.
\end{theorem}

The control may depend on $b$, but the estimate in
\eqref{eq:control-bound} is uniform for
$\lVert b\rVert_\infty\leq M$. For a general measurable set $D$, the
observability constant may also depend on how the spatial sections of
$D$ are distributed in time.

For almost every $t\in(0,T)$, define
$
D_t=\{x\in I:(t,x)\in D\}.
$
If the measure of $D_t$ is uniformly bounded from below for almost every
$t$, then the observability constant can be estimated more explicitly.

\begin{theorem}\label{thm:short-time}
Fix $L,T_0>0$, $M\geq0$, and $m\in(0,L]$. There exist constants
$C_0,C_1>0$, depending only on $L,T_0,M$, and $m$, such that, whenever
$0<T\leq T_0$, $b\in L^\infty((0,T)\times I;\mathbb R)$ satisfies
$\lVert b\rVert_\infty\leq M$, and
$D\subset(0,T)\times I$ is measurable with
\[
|D_t|\geq m
\qquad\text{for a.e. }t\in(0,T),
\]
every solution of \cref{eq:adjoint} satisfies
\begin{equation}\label{eq:short-time}
 \lVert p(0)\rVert_H^2
 \leq C_0e^{C_1/T}
 \int_D|p(t,x)|^2\,\mathrm{d}t\,\mathrm{d}x.
\end{equation}
In particular, $C_0$ and $C_1$ are uniform over all measurable sets
$D$ satisfying the above slice condition.
\end{theorem}

The slice condition in \cref{thm:short-time} gives an explicit dependence
of the observability constant on $T$, whereas \cref{thm:main} only
requires $|D|>0$. We now introduce the evolution family and the moving
finite-dimensional spaces used in the proofs.

Let $c\in L^\infty((0,T)\times I)$ be real and satisfy
$\lVert c\rVert_\infty\leq M$. Set
\[
A=-\partial_x^2,\qquad D(A)=H^2(I)\cap H_0^1(I),
\]
and denote by $U_c(t,s)$ the solution operator for
\begin{equation}\label{eq:forward}
 z'+Az+c(t,\cdot)z=0,\qquad z(s)=z_s\in H.
\end{equation}
The standard energy argument gives a unique solution
\[
z\in L^2(s,T;V)\cap H^1(s,T;V')
   \subset C([s,T];H),
\]
together with
\begin{equation}\label{eq:energy-growth}
 \lVert U_c(t,s)\rVert_{\mathcal L(H)}
 \leq e^{M(t-s)},\qquad 0\leq s\leq t\leq T.
\end{equation}

We also record the regularity needed for the differentiations below.
If
\[
v'+Av=g\in L^2(0,T;H),\qquad v(0)\in V,
\]
then
\begin{equation}\label{eq:strong-regularity}
 v\in H^1(0,T;H)\cap L^2(0,T;D(A))\cap C([0,T];V),
\end{equation}
and
\begin{equation}\label{eq:regularity-identity}
 \begin{aligned}
 \int_0^t
 \bigl(\lVert v'\rVert_H^2+\lVert Av\rVert_H^2\bigr)\,\mathrm{d}r
 +\lVert A^{1/2}v(t)\rVert_H^2
 =\lVert A^{1/2}v(0)\rVert_H^2
 +\int_0^t\lVert g\rVert_H^2\,\mathrm{d}r.
 \end{aligned}
\end{equation}
This follows by applying the identity to finite spectral truncations and
passing to the limit. In particular,
$t\mapsto\lVert A^{1/2}v(t)\rVert_H^2$ is absolutely continuous and
\[
\frac{\mathrm d}{\mathrm dt}
\lVert A^{1/2}v(t)\rVert_H^2
=2\langle Av(t),v'(t)\rangle_H
\]
for almost every $t$. For \eqref{eq:forward}, we take $g=-cv$.
For \eqref{eq:forward} with $z_s\in H$, the same regularity holds on
$[s+\varepsilon,T]$ for every $0<\varepsilon<T-s$.

We now use this evolution framework to choose the finite-dimensional
spaces for the observation argument. Let
\[
e_j(x)=\sqrt{2/L}\sin(j\pi x/L),\qquad
\lambda_j=(j\pi/L)^2,
\]
be the Dirichlet eigenfunctions and eigenvalues, and let $P_N$ be the
orthogonal projection onto
$\operatorname{span}\{e_1,\ldots,e_N\}$.
For
$
w=(\operatorname{Id}-P_N)z,
$
equation \eqref{eq:forward} gives
\begin{equation}\label{eq:fixed-complement}
 w'+Aw+(\operatorname{Id}-P_N)(cw)
 =-(\operatorname{Id}-P_N)(cP_Nz).
\end{equation}
Thus the fixed low- and high-frequency components remain coupled through
the multiplication operator $c(t,\cdot)$. In particular, the usual
Dirichlet spectral subspaces are not preserved by the evolution when
$c$ depends on $x$.

For this reason, we transport the initial Dirichlet modes by the full
evolution and set
\begin{equation}\label{eq:columns}
 f_j(t)=U_c(t,0)e_j,\qquad
 \mathcal{E}_N(t)
 =\operatorname{span}\{f_1(t),\ldots,f_N(t)\}.
\end{equation}
By \eqref{eq:strong-regularity} and the embedding
$V\hookrightarrow C([0,L])$, each $f_j$ has a jointly continuous
representative on $[0,T]\times[0,L]$, which we use for pointwise
evaluation. The composition law for $U_c$ yields
\begin{equation}\label{eq:space-invariance}
 U_c(t,s)\mathcal{E}_N(s)=\mathcal{E}_N(t).
\end{equation}
Moreover,
\[
\mathcal E_N(t)\subset \mathcal E_{N+1}(t).
\]
These spaces are determined by the evolution itself and require no time
regularity of $c$. Their invariance is the main reason for replacing the
fixed spectral subspaces by moving ones. They will be used below to
derive the finite-dimensional observation estimate and the complementary
decay estimate.

\section{Observation and decay estimates}\label{sec:subspaces}
We prove a sampling estimate on $\mathcal E_N(t)$ and a decay estimate
for the distance of a solution to this space. For the sampling
estimate, we compare the evaluation determinant of the evolved basis
$f_1(t),\ldots,f_N(t)$ with that of the first $N$ Dirichlet
eigenfunctions. This comparison is related to the total positivity of
one-dimensional diffusion kernels studied by Karlin--McGregor
\cite{KM1960} and to interpolation formulas in Barrow
\cite{Barrow1977}.

Let
\[
 \mathscr C_N
 =\{(x_1,\ldots,x_N)\in I^N:0<x_1<\cdots<x_N<L\},
 \qquad
 S_N=\sum_{j=1}^N\lambda_j.
\]
Using
\[
 \sin(j\theta)=\sin\theta\,U_{j-1}(\cos\theta),
\]
where $U_{j-1}$ is a polynomial of degree $j-1$ with nonzero leading
coefficient, we see that
$
 \det\bigl(e_j(x_i)\bigr)_{i,j=1}^N
$
is a nonzero constant times
\[
 \prod_{i=1}^N\sin(\pi x_i/L)
 \prod_{1\leq i<j\leq N}
 \bigl(\cos(\pi x_j/L)-\cos(\pi x_i/L)\bigr).
\]
Hence this determinant has a fixed sign in $\mathscr C_N$. We choose
$\sigma_N\in\{-1,1\}$ so that
\begin{equation}\label{eq:positive-determinant}
 \Phi_N(x_1,\ldots,x_N)
 =\sigma_N\det\bigl(e_j(x_i)\bigr)_{i,j=1}^N>0
 \qquad\text{in }\mathscr C_N.
\end{equation}
The same formula shows that $\Phi_N=0$ on $\partial\mathscr C_N$.
Since $-e_j''=\lambda_j e_j$, expansion of the determinant also gives
\begin{equation}\label{eq:configuration-ground-state}
 -\sum_{i=1}^N\partial_{x_i}^2\Phi_N
 =S_N\Phi_N
 \qquad\text{in }\mathscr C_N.
\end{equation}

We now compare the reference determinant $\Phi_N$ with the
corresponding determinant for the evolved basis.

\begin{lemma}\label{lem:determinant}
For every $N\geq1$, $t\in[0,T]$, and
$x=(x_1,\ldots,x_N)\in\mathscr C_N$,
\begin{equation}\label{eq:determinant-comparison}
 e^{-NMt-S_Nt}\Phi_N(x)
 \leq
 \sigma_N\det\bigl(f_j(t,x_i)\bigr)_{i,j=1}^N
 \leq
 e^{NMt-S_Nt}\Phi_N(x).
\end{equation}
In particular, for any $N$ distinct points
$x_1,\ldots,x_N\in I$, the evaluation map
\[
 v\longmapsto \bigl(v(x_1),\ldots,v(x_N)\bigr)
\]
is an isomorphism from $\mathcal E_N(t)$ onto $\mathbb R^N$.
Hence
$
 \dim\mathcal E_N(t)=N.
$
\end{lemma}
\begin{proof}
Assume first that $c$ is smooth and define
\[
 \Psi_N(t,x)
 =\sigma_N\det\bigl(f_j(t,x_i)\bigr)_{i,j=1}^N.
\]
By multilinearity of the determinant and \eqref{eq:forward},
\begin{equation}\label{eq:determinant-equation}
 \partial_t\Psi_N
 -\sum_{i=1}^N\partial_{x_i}^2\Psi_N
 +\Bigl(\sum_{i=1}^N c(t,x_i)\Bigr)\Psi_N=0
 \quad\text{in }(0,T)\times\mathscr C_N.
\end{equation}
If $x_i=x_{i+1}$, two rows of the determinant coincide. If
$x_1=0$ or $x_N=L$, one row vanishes. Hence
\[
 \Psi_N=0
 \quad\text{on }(0,T)\times\partial\mathscr C_N,
 \qquad
 \Psi_N(0,\cdot)=\Phi_N.
\]
Since
\[
 -NM\leq \sum_{i=1}^N c(t,x_i)\leq NM,
\]
and $\Phi_N$ satisfies \eqref{eq:configuration-ground-state}, the
functions
$e^{-NMt-S_Nt}\Phi_N$ and
 $e^{NMt-S_Nt}\Phi_N$
are respectively a subsolution and a supersolution of
\eqref{eq:determinant-equation}. The comparison principle on
$\mathscr C_N$ therefore gives
\eqref{eq:determinant-comparison}.

It remains to remove the smoothness assumption on $c$. Extend $c$ by
zero and mollify it to obtain smooth real coefficients $c_k$ such that
\[
 \lVert c_k\rVert_\infty\leq M,
 \qquad
 c_k\to c\quad\text{a.e. in }(0,T)\times I.
\]
Let $f_j^{(k)}$ denote the solution corresponding to $c_k$. For fixed
$j$, set
$
 w_k=f_j^{(k)}-f_j.
$
Then
\[
 w_k'+Aw_k+c_kw_k=-(c_k-c)f_j,
 \qquad
 w_k(0)=0.
\]
Since $f_j\in L^2(0,T;H)$, dominated convergence gives
\[
 (c_k-c)f_j\to0
 \qquad\text{in }L^2(0,T;H).
\]
The energy estimate then yields
\[
 w_k\to0
 \qquad\text{in }C([0,T];H).
\]
Consequently,
\[
 -c_kw_k-(c_k-c)f_j\to0
 \qquad\text{in }L^2(0,T;H).
\]
Applying \eqref{eq:regularity-identity} to
\[
 w_k'+Aw_k=-c_kw_k-(c_k-c)f_j
\]
gives
\[
 w_k\to0
 \qquad\text{in }C([0,T];V).
\]
Since $V\hookrightarrow C([0,L])$, we obtain
\[
 f_j^{(k)}\to f_j
 \qquad\text{uniformly on }[0,T]\times[0,L].
\]
For each fixed $N$, the corresponding evaluation determinants therefore
converge uniformly on $[0,T]\times[0,L]^N$. Passing to the limit in the
comparison obtained for $c_k$ proves
\eqref{eq:determinant-comparison} for bounded measurable $c$.

Finally, the lower bound in \eqref{eq:determinant-comparison} is strictly
positive for $x\in\mathscr C_N$. Hence the evaluation matrix is
invertible at any $N$ distinct interior points. It follows that the
evaluation map is an isomorphism and
\[
 \dim\mathcal E_N(t)=N.
\]
\end{proof}

The determinant estimate above gives a quantitative sampling inequality
on the moving spaces. In particular, the estimate is uniform over the
location of the observation set.
\begin{proposition}\label{prop:sampling}
Let $m\in(0,L]$. There exists $C>0$, depending only on
$L,M,T,m$, such that for every $N\geq1$, $t\in[0,T]$, every
measurable set $E\subset I$ with $|E|\geq m$, and every
$v\in\mathcal E_N(t)$,
\begin{equation}\label{eq:sampling}
 \lVert v\rVert_H
 \leq e^{CN}\lVert v\rVert_{L^2(E)}.
\end{equation}
\end{proposition}
\begin{proof}
\step{Step 1. We choose separated observation points.}
Set $R=\lVert v\rVert_{L^2(E)}$. By Chebyshev's inequality
(and directly if $R=0$), there is a measurable set $E_0\subset E$
with $|E_0|\geq m/2$ such that
\[
 |v(x)|\leq (2/m)^{1/2}R\qquad (x\in E_0).
\]
After removing neighborhoods of the endpoints of total length $m/4$,
we obtain a set
\[
 E_1\subset E_0\cap[m/8,L-m/8],
 \qquad |E_1|\geq m/4.
\]
On $E_1$,
\[
 \sin(\pi x/L)\geq\frac{m}{4L}.
\]
Let $S=\{\cos(\pi x/L):x\in E_1\}$. A change of variables gives
\[
 |S|=\int_{E_1}\frac{\pi}{L}\sin(\pi x/L)\,\mathrm dx
 \geq\frac{\pi m^2}{16L^2}.
\]
Choose a compact set $S_0\subset S$ with $|S_0|\geq |S|/2$.
The function $\xi\mapsto |S_0\cap(-\infty,\xi]|$ is continuous and
$1$-Lipschitz. We can therefore choose $\xi_1<\cdots<\xi_N$ in $S_0$
such that
\[
 |S_0\cap(-\infty,\xi_i]|=
 \frac{(i-1/2)|S_0|}{N},\qquad 1\leq i\leq N.
\]
It follows that $|\xi_i-\xi_j|\geq |S_0||i-j|/N$.
Taking $x_i\in E_1$ with $\cos(\pi x_i/L)=\xi_i$, we obtain
\[
 |\cos(\pi x_i/L)-\cos(\pi x_j/L)|
 \geq c_*\frac{|i-j|}{N},
 \qquad c_*=\frac{\pi m^2}{32L^2}.
\]

\step{Step 2. We estimate the interpolation functions.}
Let $\ell_i^c(t,\cdot)\in\mathcal E_N(t)$ be the cardinal functions
associated with these nodes, and let $\ell_i^0$ be the corresponding
functions for the first $N$ sine modes. By Cramer's rule and
\eqref{eq:determinant-comparison},
\[
 |\ell_i^c(t,x)|\leq e^{2NMt}|\ell_i^0(x)|.
\]
Indeed, each cardinal function is a quotient of evaluation
determinants. After ordering the evaluation points, the factors
$e^{-S_Nt}$ cancel between numerator and denominator. The bound at
repeated evaluation points and at the endpoints follows by continuity.
The sine-mode cardinal function has the explicit representation
\[
 \ell_i^0(x)
 =
 \frac{\sin(\pi x/L)}{\sin(\pi x_i/L)}
 \prod_{j\ne i}
 \frac{\cos(\pi x/L)-\cos(\pi x_j/L)}
      {\cos(\pi x_i/L)-\cos(\pi x_j/L)}.
\]
The lower bound on $\sin(\pi x_i/L)$ and the separation of the nodes
give
\[
 |\ell_i^0(x)|\leq
 \frac{4L}{m}
 \frac{(2N/c_*)^{N-1}}{(i-1)!(N-i)!}.
\]
Here we used $\prod_{j\ne i}|i-j|=(i-1)!(N-i)!$.
Summing over $i$ yields
\begin{align*}
 \sum_{i=1}^N|\ell_i^0(x)|
 &\leq \frac{4L}{m}\left(\frac{2N}{c_*}\right)^{N-1}
 \sum_{i=1}^N\frac{1}{(i-1)!(N-i)!}\\
 &=\frac{4L}{m}\frac{(4N/c_*)^{N-1}}{(N-1)!}
 \leq e^{C_0N},
\end{align*}
where $C_0$ depends only on $L$ and $m$. For $N\geq2$, the last
step uses $(N-1)!\geq((N-1)/e)^{N-1}$ and $N/(N-1)\leq2$;
the case $N=1$ follows by enlarging $C_0$. Therefore,
\[
 |v(x)|
 \leq
 \sum_{i=1}^N|v(x_i)|\,|\ell_i^c(t,x)|
 \leq
 (2/m)^{1/2}e^{(2MT+C_0)N}R.
\]
Taking the $L^2(I)$ norm proves \eqref{eq:sampling}, after enlarging
the constant $C$.
\end{proof}

We next estimate the component outside the moving space
$\mathcal E_N(t)$. We first record a Dirichlet-quotient estimate for
solutions of \cref{eq:forward}. This is the standard argument used in
backward uniqueness; see \cite{Kukavica2007}.

\begin{lemma}\label{lem:dirichlet-quotient}
Let $v$ solve \cref{eq:forward} with nonzero initial value in $V$.
Then $v(t)\ne0$ for every $t\in[0,T]$, and
\begin{equation}\label{eq:quotient-definition}
 \mathcal Q_v(t)
 =\frac{\lVert A^{1/2}v(t)\rVert_H^2}
        {\lVert v(t)\rVert_H^2}
\end{equation}
satisfies
\begin{equation}\label{eq:quotient-derivative}
 \mathcal Q_v'(t)\leq\frac{M^2}{2}
 \qquad\text{for almost every }t\in(0,T).
\end{equation}
\end{lemma}

\begin{proof}
On any interval where $v$ is nonzero,
\eqref{eq:strong-regularity} implies that the numerator and denominator
in \eqref{eq:quotient-definition} are absolutely continuous. Using
$
 \langle Av,v\rangle
 =\mathcal Q_v\lVert v\rVert_H^2,
$
we obtain, almost everywhere,
\begin{align*}
 \mathcal Q_v'
 =-2\frac{\lVert(A-\mathcal Q_v)v\rVert_H^2}
          {\lVert v\rVert_H^2}
   -2\frac{\langle cv,(A-\mathcal Q_v)v\rangle}
          {\lVert v\rVert_H^2}
 \leq\frac{M^2}{2}.
\end{align*}
Indeed, the last inequality follows from
$-2a^2+2Ma\leq M^2/2$ for $a\geq0$.

We also have
\[
 \frac{\mathrm d}{\mathrm dt}\log\lVert v(t)\rVert_H
 \geq-\mathcal Q_v(t)-M.
\]
Hence, for $0\leq s\leq t$ before a possible first zero,
\begin{equation}\label{eq:nonvanishing-estimate}
 \lVert v(t)\rVert_H
 \geq
 \lVert v(s)\rVert_H
 \exp\left[
   -(\mathcal Q_v(s)+M)(t-s)
   -\frac{M^2}{4}(t-s)^2
 \right].
\end{equation}
This estimate rules out a first zero by continuity. Therefore
$v(t)\ne0$ on $[0,T]$, and \eqref{eq:quotient-derivative} holds on the
whole interval.
\end{proof}

We now introduce the orthogonal projection onto the moving space.
Let $\Pi_N(t)$ be the orthogonal projection onto $\mathcal E_N(t)$,
and let $F_N(t):\mathbb R^N\to H$ have columns
$f_1(t),\ldots,f_N(t)$. By \cref{lem:determinant},
$\dim\mathcal E_N(t)=N$ for every $t\in[0,T]$. Thus
\begin{equation}\label{eq:projection-formula}
 \Pi_N(t)
 =F_N(t)G_N(t)^{-1}F_N(t)^*,
 \qquad
 G_N(t)=F_N(t)^*F_N(t).
\end{equation}
For fixed $N$, \eqref{eq:strong-regularity} gives
\[
 F_N\in H^1(0,T;\mathcal L(\mathbb R^N,H)),
 \qquad
 G_N\in H^1(0,T;\mathbb R^{N\times N}).
\]
The matrix $G_N(t)$ is continuous and positive definite on $[0,T]$,
so
\[
 \min_{t\in[0,T]}\lambda_{\min}(G_N(t))>0.
\]
Consequently, $G_N^{-1}\in H^1(0,T;\mathbb R^{N\times N})$, with
$(G_N^{-1})'=-G_N^{-1}G_N'G_N^{-1}$ almost everywhere.
For each fixed $N$ and coefficient $c$, these properties justify
differentiating the decomposition along $\mathcal E_N(t)$ and
$\mathcal E_N(t)^\perp$.

\begin{proposition}\label{prop:dissipation}
For every $N\geq1$ and $t\in[0,T]$,
\begin{equation}\label{eq:complement-coercivity}
 \lVert A^{1/2}h\rVert_H^2
 \geq
 \left(\lambda_{N+1}-\frac{M^2t}{2}\right)
 \lVert h\rVert_H^2,
 \qquad
 h\in V\cap\mathcal E_N(t)^\perp.
\end{equation}
Moreover, for $0\leq s\leq t\leq T$,
\begin{equation}\label{eq:exact-dissipation}
 \lVert(\operatorname{Id}-\Pi_N(t))U_c(t,s)\rVert_{\mathcal L(H)}
 \leq
 \exp\left[
 -(\lambda_{N+1}-M)(t-s)
 +\frac{M^2}{4}(t^2-s^2)
 \right].
\end{equation}
Choose $N_0\geq1$ such that
\[
 \lambda_{N_0+1}\geq 2M+M^2T.
\]
Then, for every $N\geq N_0$ and $0\leq s\leq t\leq T$,
\begin{equation}\label{eq:quadratic-dissipation}
 \lVert(\operatorname{Id}-\Pi_N(t))U_c(t,s)\rVert_{\mathcal L(H)}
 \leq
 e^{-\lambda_{N+1}(t-s)/2}
 \leq
 e^{-dN^2(t-s)},
 \qquad
 d=\frac{\pi^2}{2L^2}.
\end{equation}
\end{proposition}

\begin{proof}
\step{Step 1. Coercivity on the moving complement.}
Fix $t>0$ and
$0\ne h\in V\cap\mathcal E_N(t)^\perp$. Let $v$ solve
\[
 v'(r)+Av(r)+c(t-r,\cdot)v(r)=0,
 \qquad v(0)=h,\qquad 0<r<t.
\]
For $0\leq s\leq t$, duality between the forward equation and this
time-reversed equation gives
\[
 v(t-s)=U_c(t,s)^*h.
\]
Indeed, for $f\in H$, the map
$
 r\longmapsto
 \langle U_c(r,s)f,v(t-r)\rangle_H
$
is constant on $[s,t]$. Taking $s=0$ and using
$f_j(t)=U_c(t,0)e_j$, we obtain
\[
 \langle v(t),e_j\rangle_H
 =\langle h,f_j(t)\rangle_H=0,
 \qquad 1\leq j\leq N.
\]
Hence
$
 v(t)\perp\operatorname{span}\{e_1,\ldots,e_N\},
$
and therefore
$
 \mathcal Q_v(t)\geq\lambda_{N+1}.
$

The coefficient $c(t-r,\cdot)$ is still bounded by $M$, so
\cref{lem:dirichlet-quotient} applies to $v$. In particular, $v$ does
not vanish, and integrating \eqref{eq:quotient-derivative} over
$(0,t)$ gives
\[
 \frac{\lVert A^{1/2}h\rVert_H^2}{\lVert h\rVert_H^2}
 =\mathcal Q_v(0)
 \geq\mathcal Q_v(t)-\frac{M^2t}{2}
 \geq\lambda_{N+1}-\frac{M^2t}{2}.
\]
This proves \eqref{eq:complement-coercivity} for $t>0$. At $t=0$,
$\mathcal E_N(0)=\operatorname{span}\{e_1,\ldots,e_N\}$, and the
estimate follows from the usual spectral inequality.

\step{Step 2. Decay on the moving complement.}
Let $z_s\in V$ and set
\[
 z(r)=U_c(r,s)z_s,\qquad s\leq r\leq t.
\]
Using \eqref{eq:projection-formula}, write
\[
 a(r)=G_N(r)^{-1}F_N(r)^*z(r),\qquad
 w(r)=z(r)-F_N(r)a(r)
     =(\operatorname{Id}-\Pi_N(r))z(r).
\]
Thus $w(r)\perp\mathcal E_N(r)$ for every $r\in[s,t]$.

The regularity established above gives
$a\in H^1(s,t;\mathbb R^N)$ and allows us to differentiate this
decomposition. Since both $z$ and the columns of $F_N$ solve
\cref{eq:forward}, we have
\[
 z'=-(A+c)z,\qquad F_N'=-(A+c)F_N.
\]
Hence, almost everywhere on $(s,t)$,
\[
 w'=-(A+c)w-F_Na'.
\]
Now $F_N(r)a'(r)\in\mathcal E_N(r)$, while
$w(r)\in\mathcal E_N(r)^\perp$. Taking the inner product with $w(r)$
therefore gives
\begin{equation}\label{eq:distance-energy}
 \frac12\frac{\mathrm d}{\mathrm dr}\lVert w(r)\rVert_H^2
 =
 -\lVert A^{1/2}w(r)\rVert_H^2
 -\langle c(r)w(r),w(r)\rangle_H.
\end{equation}
By \eqref{eq:complement-coercivity} and
$\lVert c\rVert_\infty\leq M$,
\[
 \frac12\frac{\mathrm d}{\mathrm dr}\lVert w(r)\rVert_H^2
 \leq
 -\left(\lambda_{N+1}-M-\frac{M^2r}{2}\right)
 \lVert w(r)\rVert_H^2.
\]
Integrating from $s$ to $t$ yields
\[
 \lVert w(t)\rVert_H
 \leq
 \exp\left[
 -(\lambda_{N+1}-M)(t-s)
 +\frac{M^2}{4}(t^2-s^2)
 \right]
 \lVert w(s)\rVert_H.
\]
Since
\[
 \lVert w(s)\rVert_H
 =\lVert(\operatorname{Id}-\Pi_N(s))z_s\rVert_H
 \leq\lVert z_s\rVert_H,
\]
we obtain \eqref{eq:exact-dissipation} for $z_s\in V$.
The estimate extends to every $z_s\in H$ by density and continuity of
$U_c(t,s)$.

Finally,
\[
 t^2-s^2=(t+s)(t-s)\leq2T(t-s).
\]
If $N\geq N_0$, then
$\lambda_{N+1}\geq2M+M^2T$, and hence
\[
 \lambda_{N+1}-M-\frac{M^2T}{2}
 \geq\frac{\lambda_{N+1}}{2}.
\]
Therefore
\[
 \lVert(\operatorname{Id}-\Pi_N(t))U_c(t,s)\rVert_{\mathcal L(H)}
 \leq e^{-\lambda_{N+1}(t-s)/2}
 \leq e^{-dN^2(t-s)},
 \qquad d=\frac{\pi^2}{2L^2}.
\]
This proves \eqref{eq:quadratic-dissipation}.
\end{proof}
Compared with \eqref{eq:fixed-complement}, the moving decomposition
removes the coupling term from the energy estimate: the additional term
$F_Na'$ lies in $\mathcal E_N(r)$ and is orthogonal to $w(r)$.
This gives the closed energy estimate \eqref{eq:distance-energy} and,
in turn, the quadratic decay \eqref{eq:quadratic-dissipation}.

\section{Observability and null controllability}\label{sec:observability}

This section establishes observability and null controllability for
measurable observation sets, together with the short-time estimate under
the uniform slice condition.

\begin{lemma}\label{lem:interpolation}
Fix $m\in(0,L]$. There exist $A_0\geq1$ and $B_0>0$, depending only
on $L,M,T,m$, such that every solution of \cref{eq:forward} satisfies
\begin{equation}\label{eq:interpolation}
 \lVert z(t)\rVert_H
 \leq
 A_0e^{B_0/(t-s)}\varepsilon^{-1}
 \lVert z(t)\rVert_{L^2(E)}
 +\varepsilon\lVert z(s)\rVert_H
\end{equation}
for $0\leq s<t\leq T$, $0<\varepsilon<1$, and every measurable
$E\subset I$ with $|E|\geq m$.
If $T\leq T_0$, the constants may be chosen depending only on
$L,M,T_0,m$.
\end{lemma}

\begin{proof}
Let $C$ be the constant in \cref{prop:sampling}, and let $N_0$ and
$d$ be given by \cref{prop:dissipation}. For $N\geq N_0$, decompose
\[
 z(t)=\Pi_N(t)z(t)+(\operatorname{Id}-\Pi_N(t))z(t).
\]
By \cref{prop:sampling},
\begin{align*}
 \lVert z(t)\rVert_H
 &\leq
 e^{CN}\lVert\Pi_N(t)z(t)\rVert_{L^2(E)}
 +\lVert(\operatorname{Id}-\Pi_N(t))z(t)\rVert_H\\
 &\leq
 e^{CN}\lVert z(t)\rVert_{L^2(E)}
 +(1+e^{CN})
 \lVert(\operatorname{Id}-\Pi_N(t))z(t)\rVert_H.
\end{align*}
Using \cref{prop:dissipation} and $e^{CN}\geq1$, we obtain
\begin{equation}\label{eq:frequency-splitting}
 \lVert z(t)\rVert_H
 \leq
 e^{CN}\lVert z(t)\rVert_{L^2(E)}
 +2e^{CN-dN^2(t-s)}\lVert z(s)\rVert_H.
\end{equation}

Set $\delta=t-s$ and $\ell=\log(2/\varepsilon)$, and choose
$
 N=
 \left\lceil
 \max\left\{
 N_0,\frac{2C}{d\delta},
 \sqrt{\frac{2\ell}{d\delta}}
 \right\}
 \right\rceil.
$
Then
\[
 CN-d\delta N^2
 \leq-\frac{d\delta N^2}{2}
 \leq-\ell,
\]
so the second term in \eqref{eq:frequency-splitting} is bounded by
$\varepsilon\lVert z(s)\rVert_H$.

Moreover,
\begin{align*}
 CN
 \leq
 C(N_0+1)+\frac{2C^2}{d\delta}
 +C\sqrt{\frac{2\ell}{d\delta}}
 \leq
 C(N_0+1)+\frac{5C^2}{2d\delta}+\ell.
\end{align*}
Since $e^\ell=2/\varepsilon$, \eqref{eq:interpolation} follows with
\[
 A_0=2e^{C(N_0+1)},
 \qquad
 B_0=\frac{5C^2}{2d}.
\]
For $T\leq T_0$, the constants in
\cref{prop:sampling,prop:dissipation} may be chosen in terms of
$L,M,T_0,m$, which gives the final statement.
\end{proof}

We now turn to observability on measurable spacetime sets.
The resulting estimate is uniform for all coefficients satisfying
$\lVert c\rVert_\infty\leq M$.

\begin{proposition}\label{prop:forward-observability}
Let $\mathcal D\subset(0,T)\times I$ be measurable with
$|\mathcal D|>0$. There exists a constant
$C_*=C_*(L,T,M,\mathcal D)>0$ such that every solution of
\cref{eq:forward} satisfies
\begin{equation}\label{eq:mixed-observability}
 \lVert z(T)\rVert_H
 \leq
 C_*\int_0^T
 \lVert z(t)\rVert_{L^2(\mathcal D_t)}\,\mathrm dt.
\end{equation}
Consequently,
\begin{equation}\label{eq:forward-L2}
 \lVert z(T)\rVert_H^2
 \leq
 TC_*^2
 \int_{\mathcal D}|z(t,x)|^2\,\mathrm dt\,\mathrm dx.
\end{equation}

Moreover, fix $T_0>0$ and $m\in(0,L]$. There exist
$C_2,C_3>0$, depending only on $L,T_0,M,m$, such that, whenever
$0<T\leq T_0$ and
$|\mathcal D_t|\geq m$ for almost every $t\in(0,T)$,
every solution of \cref{eq:forward} satisfies
\begin{equation}\label{eq:forward-short-time}
 \lVert z(T)\rVert_H
 \leq
 C_2e^{C_3/T}
 \int_0^T
 \lVert z(t)\rVert_{L^2(\mathcal D_t)}\,\mathrm dt.
\end{equation}
\end{proposition}

\begin{proof}
\step{Step 1. We choose the observation times.}
By Fubini's theorem, there exists $m\in(0,L]$ such that
$
 F=\{t\in(0,T):|\mathcal D_t|\geq m\}
$
has positive measure. Choose a Lebesgue density point
$\ell\in F\cap(0,T)$. There exists
$r_0\in(0,\min\{1,T-\ell\})$ such that
\begin{equation}\label{eq:density-property}
 |(\ell,\ell+r)\setminus F|
 \leq \frac r{16},
 \qquad 0<r\leq r_0.
\end{equation}
For $j\geq0$, set
\begin{equation}\label{eq:telescoping-times}
 r_j=r_0(3/4)^j,\qquad
 t_j=\ell+r_j,\qquad
 J_j=(\ell+7r_j/8,\ell+r_j).
\end{equation}
The intervals $J_j$ are pairwise disjoint, and
\begin{equation}\label{eq:observation-time-mass}
 |F\cap J_j|\geq\frac{r_j}{16},
 \qquad
 t-t_{j+1}\geq\frac{r_j}{8}
 \quad (t\in J_j).
\end{equation}
Set
\[
 H_j=\lVert z(t_j)\rVert_H,
 \qquad
 g(t)=\lVert z(t)\rVert_{L^2(\mathcal D_t)}.
\]

\step{Step 2. We derive a discrete estimate.}
Let $A_0,B_0$ be given by \cref{lem:interpolation}. For almost every
$t\in F\cap J_j$, we apply \eqref{eq:interpolation} with
$s=t_{j+1}$ and $E=\mathcal D_t$. Using
\eqref{eq:energy-growth} from $t$ to $t_j$, and then choosing
$\varepsilon=\eta e^{-MT}$ with $0<\eta<1$, gives
\[
 H_j
 \leq
 \eta H_{j+1}
 +A_0e^{2MT+8B_0/r_j}\eta^{-1}g(t).
\]
Averaging over $F\cap J_j$ and using
\eqref{eq:observation-time-mass}, we obtain
\begin{equation}\label{eq:discrete-recursion}
 H_j
 \leq
 \eta H_{j+1}
 +e^{K/r_j}\eta^{-1}
 \int_{F\cap J_j}g(t)\,\mathrm dt,
\end{equation}
Here we may take
\[
 K=8B_0+\log(16A_0)+2MT+1,
\]
which is independent of $j$. Indeed, $0<r_j<1$ and
$\log(1/r_j)\leq 1/r_j$ give
\[
 \frac{16A_0}{r_j}e^{2MT+8B_0/r_j}\leq e^{K/r_j}.
\]

\step{Step 3. We telescope the discrete estimates.}
Choose
$
 \eta=e^{-K/r_j}
$
in \eqref{eq:discrete-recursion} and multiply the result by
$e^{-3K/r_j}$. Since $r_{j+1}=3r_j/4$,
\[
 e^{-3K/r_j}\eta=e^{-3K/r_{j+1}}.
\]
The coefficient of the observation integral becomes
$e^{-K/r_j}\leq1$. Hence
\[
 e^{-3K/r_j}H_j
 -e^{-3K/r_{j+1}}H_{j+1}
 \leq
 \int_{F\cap J_j}g(t)\,\mathrm dt.
\]
Summing from $j=0$ to $n$ and using the disjointness of the $J_j$
gives
\[
 e^{-3K/r_0}H_0
 -e^{-3K/r_{n+1}}H_{n+1}
 \leq
 \int_F g(t)\,\mathrm dt.
\]
By \eqref{eq:energy-growth},
\[
 H_j\leq e^{MT}\lVert z(0)\rVert_H,
\]
while $r_j\to0$. Thus the second term on the left tends to zero.
Applying \eqref{eq:energy-growth} once more from $t_0$ to $T$, we get
\begin{equation}\label{eq:explicit-mixed-constant}
 \lVert z(T)\rVert_H
 \leq
 e^{MT+3K/r_0}
 \int_F g(t)\,\mathrm dt.
\end{equation}
Since $F\subset(0,T)$, this proves \eqref{eq:mixed-observability}.
Cauchy--Schwarz and Fubini's theorem then give
\eqref{eq:forward-L2}. The parameters $m,\ell,r_0$ depend only on
$\mathcal D$, so the estimate is uniform for all
$\lVert c\rVert_\infty\leq M$.

\step{Step 4. We obtain the short-time estimate.}
Assume now that
\[
 |\mathcal D_t|\geq m
 \qquad\text{for almost every }t\in(0,T).
\]
Then $F=(0,T)$ up to a null set. We may take
\[
 \ell=\frac T4,
 \qquad
 r_0=\min\left\{\frac T2,\frac12\right\}.
\]
For $0<T\leq T_0$, choose $A_0,B_0$ in
\cref{lem:interpolation} uniformly in terms of $L,M,T_0,m$.
The constant $K$ in \eqref{eq:discrete-recursion} may then also be
chosen depending only on these parameters.

Since
$
 \frac1{r_0}\leq2+\frac2T,
$
\eqref{eq:explicit-mixed-constant} yields
\[
 \lVert z(T)\rVert_H
 \leq
 C_2e^{C_3/T}
 \int_0^T
 \lVert z(t)\rVert_{L^2(\mathcal D_t)}\,\mathrm dt,
\]
where $C_2,C_3>0$ depend only on $L,M,T_0,m$. This proves
\eqref{eq:forward-short-time}.
\end{proof}

We now return to the adjoint problem and complete the proof of the main
observability and null controllability results.

\begin{proof}[Proof of \cref{thm:main,thm:short-time}]
\step{Step 1. We prove the adjoint observation inequalities.}
Let $p$ solve \cref{eq:adjoint} and define
\begin{equation}\label{eq:time-reversal}
 z(s,x)=p(T-s,x),\qquad
 c(s,x)=b(T-s,x),\qquad
 \mathcal D=\{(s,x):(T-s,x)\in D\}.
\end{equation}
Then $z$ solves \eqref{eq:forward} with $z(0)=q$,
$\lVert c\rVert_\infty=\lVert b\rVert_\infty$, and
$|\mathcal D|=|D|$. Moreover,
\[
 \lVert z(T)\rVert_H=\lVert p(0)\rVert_H,
 \qquad
 \int_{\mathcal D}|z(s,x)|^2\,\mathrm ds\,\mathrm dx
 =
 \int_D|p(t,x)|^2\,\mathrm dt\,\mathrm dx.
\]
Thus \eqref{eq:forward-L2} gives \eqref{eq:observability} with
$C=TC_*^2$.

If $|D_t|\geq m$ for almost every $t$, then
$|\mathcal D_s|=|D_{T-s}|\geq m$ for almost every $s$. Hence
\eqref{eq:forward-short-time} and Cauchy--Schwarz give
\begin{align*}
 \lVert p(0)\rVert_H^2
 =\lVert z(T)\rVert_H^2
 \leq
 C_2^2e^{2C_3/T}
 \left(
   \int_0^T\lVert z(s)\rVert_{L^2(\mathcal D_s)}\,\mathrm ds
 \right)^2
 \leq
 TC_2^2e^{2C_3/T}
 \int_{\mathcal D}|z(s,x)|^2\,\mathrm ds\,\mathrm dx.
\end{align*}
Since $T\leq T_0$, this proves \eqref{eq:short-time} with
 $C_0=T_0C_2^2$ and $C_1=2C_3$.

\step{Step 2. We construct the null control.}
Let $p_q$ denote the solution of \cref{eq:adjoint} with terminal value
$q\in H$. On the subspace
$
 \mathcal X=\{p_q|_D:q\in H\}\subset L^2(D),
$
define
\[
 \Lambda(p_q|_D)
 =-\langle y_0,p_q(0)\rangle_H.
\]
If $p_q|_D=p_{\widetilde q}|_D$, then
\eqref{eq:observability} applied to $p_q-p_{\widetilde q}$ gives
$p_q(0)=p_{\widetilde q}(0)$. Hence $\Lambda$ is well defined.
Moreover,
\[
 |\Lambda(p_q|_D)|
 \leq
 \lVert y_0\rVert_H\lVert p_q(0)\rVert_H
 \leq
 C^{1/2}\lVert y_0\rVert_H
 \lVert p_q\rVert_{L^2(D)}.
\]
The Hahn--Banach theorem extends $\Lambda$ to a bounded linear
functional on $L^2(D)$ with the same norm. By the Riesz representation
theorem, there exists $u\in L^2(D)$ such that
\[
 \lVert u\rVert_{L^2(D)}
 \leq C^{1/2}\lVert y_0\rVert_H
\]
and
\[
 \int_Du p_q\,\mathrm dt\,\mathrm dx
 =-\langle y_0,p_q(0)\rangle_H,
 \qquad q\in H.
\]
The duality identity for the controlled and adjoint equations gives
\begin{equation}\label{eq:control-duality}
 \langle y(T),q\rangle_H
 =
 \langle y_0,p_q(0)\rangle_H
 +\int_Du p_q\,\mathrm dt\,\mathrm dx
 =0
 \qquad(q\in H).
\end{equation}
Therefore $y(T)=0$, and \eqref{eq:control-bound} follows.

\step{Step 3. We prove necessity of positive measure.}
If $|D|=0$, then $\mathbf 1_Du=0$ almost everywhere for every
$u\in L^2(D)$, so the controlled equation reduces to the uncontrolled
one. Take $y_0=e_1$. Since $e_1\in V$ is nonzero,
\cref{lem:dirichlet-quotient} applied with $c=b$ shows that the
corresponding solution satisfies
\[
 y(T)\ne0.
\]
Thus this initial state cannot be driven to zero. Hence null
controllability requires $|D|>0$.
\end{proof}

The mixed-norm observability estimate also yields controls that are
bounded in time.

\begin{corollary}\label{cor:bounded-control}
Let $D\subset(0,T)\times I$ be measurable with $|D|>0$, and let
$M\geq0$. There exists a constant $C_\infty=C_\infty(L,T,M,D)>0$
such that, for every real-valued $b\in L^\infty((0,T)\times I)$ with
$\lVert b\rVert_\infty\leq M$ and every $y_0\in H$,
\cref{eq:controlled} admits a null control, extended by zero outside
$D$, satisfying
\begin{equation}\label{eq:bounded-control}
 u\in L^\infty(0,T;H),
 \qquad
 \lVert u\rVert_{L^\infty(0,T;H)}
 \leq C_\infty\lVert y_0\rVert_H.
\end{equation}
Under the assumptions of \cref{thm:short-time}, one may choose
$
 C_\infty\leq C_2e^{C_3/T},
$
where $C_2,C_3>0$ depend only on $L,T_0,M,m$.
\end{corollary}

\begin{proof}
After time reversal, \eqref{eq:mixed-observability} gives
\[
 \lVert p_q(0)\rVert_H
 \leq
 C_\infty
 \lVert \mathbf 1_Dp_q\rVert_{L^1(0,T;H)}.
\]
Hence the functional
$
 \mathbf 1_Dp_q
 \longmapsto
 -\langle y_0,p_q(0)\rangle_H
$
is well defined on
$\{\mathbf 1_Dp_q:q\in H\}\subset L^1(0,T;H)$ and satisfies
\[
 \bigl|\langle y_0,p_q(0)\rangle_H\bigr|
 \leq
 C_\infty\lVert y_0\rVert_H
 \lVert\mathbf 1_Dp_q\rVert_{L^1(0,T;H)}.
\]
By the Hahn--Banach theorem, this functional extends to
$L^1(0,T;H)$ with the same norm. Since $H$ is a separable Hilbert
space, we have the isometric identification
$(L^1(0,T;H))^*=L^\infty(0,T;H)$ under the integral pairing.
The extension is therefore represented by some
$v\in L^\infty(0,T;H)$. Thus
\[
 \lVert v\rVert_{L^\infty(0,T;H)}
 \leq C_\infty\lVert y_0\rVert_H
\]
and
\[
 \int_0^T
 \langle v(t),\mathbf 1_{D_t}p_q(t)\rangle_H\,\mathrm dt
 =
 -\langle y_0,p_q(0)\rangle_H.
\]

Set
$
 u(t,x)=\mathbf 1_D(t,x)v(t,x).
$
Then
\[
 \lVert u\rVert_{L^\infty(0,T;H)}
 \leq
 \lVert v\rVert_{L^\infty(0,T;H)}
 \leq
 C_\infty\lVert y_0\rVert_H,
\]
and $u\in L^2((0,T)\times I)$ since $T<\infty$. Moreover,
\[
 \int_Du p_q\,\mathrm dt\,\mathrm dx
 =
 -\langle y_0,p_q(0)\rangle_H.
\]
The duality identity \eqref{eq:control-duality} therefore gives
$y(T)=0$.

Under the slice assumptions of \cref{thm:short-time},
\eqref{eq:forward-short-time} yields
$C_\infty\leq C_2e^{C_3/T}$ with $C_2,C_3$ depending only on
$L,T_0,M,m$.
\end{proof}

We conclude with two examples showing that the observation set need
not contain any open spatial region. Let $K\subset I$ be a closed
nowhere-dense set with $|K|>0$. Then
\[
 D=(0,T)\times K
\]
satisfies the assumptions of \cref{thm:main,thm:short-time}, with
$|D_t|=|K|$ for every $t\in(0,T)$.

Observation sets need not have product structure either. Let
$\mathbb T=\mathbb R/\mathbb Z$ be equipped with normalized Lebesgue
measure, and let $K\subset\mathbb T$ be closed and nowhere dense with
$0<|K|<1$. Define
\begin{equation}\label{eq:moving-cantor}
 D_K=
 \left\{
 (t,x)\in(0,T)\times(0,L):
 \frac{x}{L}-\frac{t}{T}\pmod 1\in K
 \right\}.
\end{equation}
For every $t\in(0,T)$,
\[
 |(D_K)_t|=L|K|,
\]
while $(D_K)_t$ contains no open interval. Hence
\cref{thm:short-time} applies with $m=L|K|$, with constants uniform
over all real potentials satisfying $\lVert b\rVert_\infty\leq M$.

\begin{proposition}\label{prop:nonproduct}
The set $D_K$ in \cref{eq:moving-cantor} contains no product of two
positive-measure sets. More precisely, there do not exist measurable
sets $F\subset(0,T)$ and $E\subset(0,L)$ with $|F|,|E|>0$ such that
\[
 |(F\times E)\setminus D_K|=0.
\]
\end{proposition}

\begin{proof}
Suppose that such sets $F$ and $E$ exist. After the changes of variables
$\alpha=t/T$ and $\beta=x/L$, we obtain measurable sets
$F_0,E_0\subset\mathbb T$ of positive measure such that
$
 \beta-\alpha\in K
$
for almost every $(\alpha,\beta)\in F_0\times E_0$.

Define
\[
 h(s)=\int_{\mathbb T}
 \mathbf 1_{F_0}(\alpha)\mathbf 1_{E_0}(s+\alpha)\,\mathrm d\alpha.
\]
The function $h$ is nonnegative and continuous, since translations are
continuous in $L^2(\mathbb T)$. Moreover,
\[
 \int_{\mathbb T}h(s)\,\mathrm ds
 =|F_0||E_0|>0.
\]
On the other hand, the assumed product containment gives
\[
 \int_{\mathbb T\setminus K}h(s)\,\mathrm ds=0.
\]
Since $h\geq0$ and $\mathbb T\setminus K$ is open, continuity implies
that $h=0$ on $\mathbb T\setminus K$. This set is dense because $K$ is
nowhere dense. Hence $h=0$ on all of $\mathbb T$, contradicting the
strictly positive integral above.
\end{proof}

For instance, when $M>0$, the potential
\[
 b(t,x)=M\mathbf 1_{D_K}(t,x)
\]
is real, bounded, and discontinuous along the moving set $D_K$.
Both main theorems apply to this coefficient and to the observation set
$D_K$.

\end{document}